\documentclass[11pt]{article}
\usepackage{amsmath}
\usepackage{amssymb}%,amsthm}
\usepackage{amscd,euscript}
\usepackage{latexsym}
\usepackage{theorem}
\usepackage{epsfig}
\usepackage{amsfonts,microtype}
\usepackage{xypic,xspace}
\usepackage{bbm,ifpdf}
\usepackage[margin=3.2cm]{geometry}
\usepackage{pxfonts}

\ifpdf
  \usepackage{pdfsync,hyperref}
\fi

\newtheorem{theorem}{Theorem}
\newtheorem{lemma}[theorem]{Lemma}
\newtheorem{proposition}[theorem]{Proposition}
\newtheorem{corollary}[theorem]{Corollary}

{
\theorembodyfont{\rmfamily}
\theoremstyle{plain}
\newtheorem{definition}[theorem]{Definition}

\newtheorem{remark}[theorem]{Remark}
}

\newcommand{\qed}{\hfill \mbox{$\Box$}\medskip\newline}
\newenvironment{proof}{\noindent {\bf Proof:}}{\qed \par}
\newenvironment{prooftr}{\noindent {\bf Proof of Proposition \ref{trans}:}}{}

\newcommand{\im}{\operatorname{im}}

\renewcommand{\dim}{\operatorname{dim}}
\newcommand{\Ann}{\operatorname{Ann}}

\newcommand{\Z}{\mathbb{Z}}

\newcommand{\vp}{\varphi}

\newcommand{\C}{\mathbb{C}}

\newcommand{\la}{\leftarrow}

\newcommand{\fg}{{\mathfrak{g}}}
\newcommand{\fn}{{\mathfrak{n}}}
\newcommand{\fm}{{\mathfrak{m}}}
\newcommand{\fb}{{\mathfrak{b}}}
\newcommand{\ft}{{\mathfrak{t}}}
\newcommand{\fp}{{\mathfrak{p}}}
\newcommand{\fh}{{\mathfrak{h}}}
\newcommand{\fl}{{\mathfrak{l}}}
\newcommand{\fq}{{\mathfrak{q}}}

\renewcommand{\la}{\lambda}

\renewcommand{\a}{\alpha}

\newcommand{\Hom}{\operatorname{Hom}}

\newcommand{\cO}{\mathcal{O}}
\renewcommand{\cL}{\mathcal{L}}

\newcommand{\hcO}{\widehat{\cO}}
\newcommand{\pcO}{\cO'}

\newcommand{\cS}{\mathcal{S}}

\newcommand{\cM}{\mathcal{M}}
\newcommand{\cW}{\mathcal{W}}
\renewcommand{\cD}{\mathcal{D}}

\newcommand{\gr}{\operatorname{gr}}

\newcommand{\soc}{\operatorname{soc}}

\newcommand{\cG}{\mathcal{G}}
\newcommand{\fX}{\mathfrak{X}}

\newcommand{\sQ}{\mathcal{Q}}

\renewcommand{\cH}{\mathcal{H}}
\newcommand{\excise}[1]{}
\newcommand{\Milicic}{Mili\v ci\'c\xspace}
\newcommand{\Mirkovic}{Mirkovi\'c\xspace}
\newcommand{\al}{\alpha}

\begin{document}

\noindent {\Large \bf
Singular blocks of parabolic category $\cO$ and finite W-algebras}
\bigskip\\
{\bf Ben Webster}\footnote{Supported by an NSF Postdoctoral Research
  Fellowship and  by  NSA grant H98230-10-1-0199.}\\
Department of Mathematics, University of Oregon, Eugene, OR 97403\\
\textsf{bwebster@uoregon.edu}

\bigskip
{\small

\begin{quote}
\noindent {\em Abstract.}
We show that each integral infinitesimal block of parabolic category $\cO$
(including singular ones) for a semi-simple Lie algebra can be realized
as a full subcategory of a modified category $\cO$ over a finite
W-algebra for the same Lie algebra.

The nilpotent used to construct this finite W-algebra is determined by
the central character of the block, and the subcategory taken is that
killed by a particular two-sided ideal depending on the original parabolic. The
equivalences in question are induced by those defined by \Milicic-Soergel and
Losev.

We also give a proof of a result of some independent interest: the
singular blocks of parabolic category $\cO$ can be geometrically
realized as ``partial Whittaker sheaves'' on partial flag varieties.
\end{quote}
}
\bigskip

One of the central pillars of modern representation theory is the
category $\cO$ of representations of a semi-simple Lie algebra $\fg$ defined by Bernstein, Gelfand and Gelfand \cite{BGG} and
its generalizations.  Among the remarkable phenomena which have arisen
in the study of these categories is the frequent appearance of
unexpected equivalences of categories. Our aim in this paper is
describe just such an equivalence; in particular, we define an equivalence
between a singular block of category $\cO$ and a modification of the 
category $\cO$ over the finite W-algebra defined by Brundan, Goodwin
and Kleshchev \cite{BGK}.

Let us be more precise.  Recall that if $\chi:Z(U(\fg))\to \C$ is a character of the center of
$U(\fg)$ then a $U(\fg)$-module $M$ has {\bf generalized central
  character} $\chi$ if $z-\chi(z)$ acts locally nilpotently on $M$ for
all $z\in Z(U(\fg))$, and that an {\bf infinitesimal block} of a
category of $\fg$-module is the subcategory of modules with fixed
generalized central character.

In this paper, we show that any integral infinitesimal block of
parabolic category $\cO$ for a semi-simple Lie algebra $\fg$ has a
full and faithful functor $\Psi$ to the category of representations over a
finite W-algebra for $\fg$ corresponding to a nilpotent $e_\chi$ which depends
on the central character $\chi$ of the block.
The functor $\Psi$ is obtained by restricting equivalences defined in previous work by
\Milicic-Soergel \cite{MiS} on representations of Lie algebras and
Losev \cite{LosO} on finite W-algebras.  Thus, the new ingredient of
this paper is to trace how various finiteness properties are
transformed under these equivalences. 

\medskip

The essential image of $\Psi$ (that is, the full subcategory it defines an
equivalence onto) has a flavor similar to
the category $\cO$ for the W-algebra already defined by Brundan,
Goodwin and Kleshchev \cite{BGK}, but it is subtly different, and
appears to be new to the literature.  The reason for its importance is
that it is a special case of a definition associating a ``category
$\cO$'' to any conical symplectic resolution of singularities; in this
case, the resolution in question is 
\[
\fX^Q_{e_\chi}=\{(gQ,x)\in G/Q\times \fg\,\vert\,
x\in \mathrm{Ad}_g(\fq^\perp)\cap \cS\}
\] where $G$ is the adjoint
group of $\fg$ and $\cS$ is the Slodowy
slice to the nilpotent orbit $G\cdot e$. 
The variety $\fX^Q_{e_\chi} $ is symplectic and has a deformation contraction to
the Lagrangian subvariety of elements where $x=e_\chi$, which is a
Spaltenstein variety.  Unfortunately, this variety does not seem
to have an agreed-upon name; we will refer to it as an {\bf S3
  variety}\footnote{``S3'' stands for ``S(lodowy)-S(paltenstein)-S(pringer).''  We
  thank/blame Nick Proudfoot for this coinage.}.

In particular, this theorem provides a direct link between the
geometry of the varieties $\fX^Q_{e_\chi}$ and category $\cO$, whose existence
has been suggested by results of Brundan \cite{Bru06,Bru08} and
Stroppel \cite{Str06b} relating the cohomology of Spaltenstein
varieties (and thus $\fX_{e_\chi}^Q$) to
the centers of singular blocks of parabolic category $\cO$.  The
general properties of such categories and their connection to geometry
will be explored more fully in future work by the author, Braden,
Licata and Proudfoot \cite{BLPWgco}.

Now, we turn to giving a more precise description of the essential
image.  Consider the category $\mathcal{C}$ of modules $M$ the W-algebra $\cW_e$
attached to a nilpotent $e$ such that
\begin{itemize}
\setlength{\itemsep}{-1pt}
\item $M$ has the same central character as the trivial module (in
  particular, the center acts semi-simply on them) and
\item every simple composition factor of $M$ lies in the category
  $\cO$ defined by  \cite{BGK}.
\end{itemize}
This category has two subtle differences from the category $\cO$'s defined in \cite{BGG} for $U(\fg)$ and in \cite{BGK} for the W-algebra:
\begin{itemize}
\setlength{\itemsep}{-1pt}
\item The category $\mathcal{C}$ only contains modules where the center acts semi-simply.
\item We have no analogue of the condition that the Cartan act
  semi-simply, which is imposed in \cite{BGG} (an analogue of this
  semi-simplicity is also imposed in \cite{BGK}).
\end{itemize}
Let $\chi$ be an integral central character of $U(\fg)$, and
$\cO_\chi$ the associated block of category $\cO$; each such character has an associated conjugacy class of Levi subgroups which measure how ``singular'' the block is.  The Weyl group of the associated Levi is the stabilizer of the highest weight of any highest weight module in the corresponding infinitesimal block under the dot action.  Let $e_\chi$ be a regular nilpotent in the Levi corresponding to $\chi$. 

\renewcommand{\thetheorem}{}
\setcounter{theorem}{-1}

\begin{theorem}
~
  \begin{itemize}
\itemsep=.01\itemsep
  \item The block $\cO_\chi$ is equivalent to the category
    $\mathcal{C}$ for the W-algebra
    $\cW_{e_\chi}$ of the nilpotent $e_\chi$ via the functor $\Psi$.
\item  The functor $\Psi$ induces an equivalence between the subcategory of modules locally finite
    for a parabolic $\fq$ in $\cO_\chi$ and the subcategory in
    $\mathcal{C}$ of modules killed by a particular
    two-sided ideal $J_\fq$ in the W-algebra. 
\item The quotient
    $\cW_{e_\chi}/J_\fq$ is a quantization of the variety
    $\fX^Q_{e_\chi}$;  that is, for a particular filtration, we have an
    isomorphism $\gr(\cW_\chi/J_\fq)\cong \C[\fX^Q_{e_\chi}]$.
  \end{itemize}

\end{theorem}

The proof of the theorem above is organized into three sections. Section
\ref{lie} covers preparatory results on category $\cO$ and
Harish-Chandra bimodules; it is likely that many of the results
therein are familiar to experts, but they do not seem to have been
written down together in the necessary form anywhere in the literature. Section
\ref{W-alg} describes the relationship of these results to the
representation theory of the finite W-algebra, and completes the
proof of the first two parts of the theorem above.

Finally, in Section \ref{geometry}, we also discuss briefly the
relationship of these results to geometry.  We show that the ideals
$J_\fq$ appearing above have a natural geometric interpretation which
gives a proof of the last part of the main theorem.

Furthermore, we give a short proof of a result which has circulated as
a folk theorem: the singular integral infinitesimal blocks of
parabolic category $\cO$ can be geometrically realized as
$(N,f)$-equivariant (the so-called ``partial Whittaker'') $D$-modules
or perverse sheaves on $G/Q$ for a particular character $f\colon N\to
\mathbb{G}_a$.  This is a classical theorem when $f$ is generic (going
back to a paper of Kostant \cite{Kos}), and the $Q=B$ case can be
derived from Bezrukavnikov and Yun \cite[Theorem 4.4.1]{BY} (as we
shall explain) but we know of
no proof in the literature of the general case.

\subsection*{Acknowledgments}

I would like to thank Tom Braden, Tony Licata and Nick Proudfoot both
for useful discussions during the development of this paper, and
extremely helpful comments on earlier versions of this paper; Jim
Humphreys, Ivan Losev and an anonymous referee for very useful comments on an
earlier version of the paper; Catharina Stroppel, Roman Bezrukavnikov
and Ivan \Mirkovic for some valuable education on Harish-Chandra
bimodules and Whittaker sheaves; and the Mathematische
Forschungsinstitut Oberwolfach for hospitality while this paper was
written.

\subsection{Lie theory}
\label{lie}

First, we fix notation.  Let $\fg\supset \fq\supset \fb=\fh\oplus \fn$
be a semi-simple complex Lie algebra, a parabolic, a Borel, and a
chosen decomposition of $\fb$ as a Cartan and its nilpotent radical.
Let $G\supset Q\supset B = HN$ be connected groups with these Lie
algebras (which form of $G$ we take is irrelevant to our purposes).
Let $Z$ denote the center of the universal enveloping algebra 
$U(\fg)$.
\renewcommand{\thetheorem}{\arabic{theorem}}

\begin{definition}
We let $\hcO$ be the category of
representations of $\fg$ where $\fn$ and $Z$ act locally finitely.
\end{definition}
This category has a natural infinitesimal block decomposition
$\hcO=\bigoplus\hcO(\chi,f)$ where the sum runs over characters
$\chi\colon Z\to \C$ and $f\colon\fn \to \C$ and $\hcO(\chi,f)$ is the
subcategory where all the irreducible constituents of the restriction
to $U(\fn)\otimes Z$ are $f\otimes \chi$.  For any integral dominant weight
$\la$, we let $\chi_\la$ be the character of $Z$ acting on the unique
irreducible $\fg$-module of highest weight $\la$.

The category $\hcO$ contains two very natural subcategories:
\begin{itemize}
\itemsep=.01\itemsep
\item $\cO$ is the subcategory of modules on which $\fh$ acts semi-simply.
\item $\pcO$ is the subcategory of modules on which $Z$ acts semi-simply.
\end{itemize}
These have the same infinitesimal block decomposition as above.  We
note that $\cO$ is the most commonly studied of these categories: in our notation $\bigoplus_\chi\cO(\chi,0)$
 is precisely the category $\cO$ originally defined by
Bernstein-Gelfand-Gelfand \cite{BGG}, and $\cO(\chi,0)$ its
infinitesimal block for the central character $\chi$.

It's worth noting that for non-integral $\chi$, the infinitesimal
blocks $\cO(\chi,0)$ are not necessarily blocks in the abstract sense;
they may have further decompositions as the sum of orthogonal
subcategories.  For several of the categories we consider, the
abstract block decompositions can actually be quite subtle and
complicated.

We will also wish to consider the parabolic version of category $\cO$.
Let $\cO^\fq(\chi,0)$ be the subcategory of $\cO(\chi,0)$ consisting 
of modules also locally finite for the action of $\fq$.  One
particular special object in $\cO^\fq(\chi_0,0)$ is the dominant
parabolic Verma module
$M^{\fq}=\operatorname{Ind}_{U(\fq)}^{U(\fg)}\mathbbm{1}$, where
$\mathbbm{1}$ is the trivial representation of $\fq$.

For the entirety of the paper, we will always use $f$ and $\mu$ to
denote a Lie algebra map $f\colon\fn\to
\C$ and a weight $\mu$ of $\fg$ such that if $\Delta_f$ is the set of simple roots $\a$ whose root spaces
are not killed by $f$, then 
$\a^\vee(\mu)=-1$ for all $\a\in \Delta_f$ and $\a^\vee(\mu)\in
\Z\setminus\{-1\}$ for all $\a\notin \Delta_f$.  In particular, note
that $f=0$ and $\mu=0$ are compatible in the sense above; this will be
an important special case for us.
Let $\fp_f$ be the parabolic generated by the
standard Borel $\fb$ and the negative root spaces $\fg_{-\a}$ for
$\a\in\Delta_f$.

The starting point for us is an equivalence of categories
$\Phi_\mu\colon\hcO(\chi_\mu,0)\to \hcO(\chi_0,f)$ constructed by
\Milicic and Soergel \cite[Theorem 5.1]{MiS}.  This functor makes
sense for any integral value of $\mu$ 
for any compatible choice of $f$, so for simplicity we leave $f$ out
from the notation.  Our main result of this
section will be to understand the effect of the functor $\Phi_\mu$ on
subcategories of interest to us. 

One case of special interest is the functor
$\Phi_0\colon\hcO(\chi_0,0)\to\hcO(\chi_0,0)$.  This functor was
considered in earlier work of Soergel \cite{Soe86} to construct an
equivalence between $\cO(\chi_0,0)$ and $\pcO(\chi_0,0)$.
We let $I_\fq=\Ann_{U(\fg)}\Phi_0(M^{\fq})$ and let
$\pcO_\fq(\chi_0,f)\subset \pcO(\chi_0,f)$ denote the subcategory of modules killed by
$I_\fq$. We
note that $I_\fq\cap Z=\ker \chi_0$, since $\Phi_0(M^{\fq})$ is a
quotient of a Verma module.  Thus any object of $\hcO(\chi_0,f)$ killed
by $I_\fq$ is automatically a semi-simple $Z$ module, and therefore in
$\pcO(\chi_0,f)$. 

\begin{proposition}\label{trans}
  The subalgebra $\fq$ acts locally finitely and $\fh$ acts semi-simply on
  a module $X\in \hcO(\chi_\mu,0)$ if and only if
  $\Ann(\Phi_\mu(X))\supset I_\fq$; that is, we have an induced
  equivalence $\Phi_\mu\colon\cO^\fq(\chi_\mu,0)\to \pcO_\fq(\chi_0,f)$.

  In particular, setting $\fq=\fb$, we have that $\fh$ acts
  semi-simply on $X\in \hcO(\chi_\mu,0)$ if and only if $Z$ acts
  semi-simply on $\Phi_\mu(X)$; that is, we have an induced equivalence
  $\Phi_\mu\colon\cO(\chi_\mu,0)\to \pcO(\chi_0,f)$.
\end{proposition}
Most of the ideas necessary for the proof of this fact are already in
the literature, which for the convenience of the reader, we now
collect into a pair of lemmata.
\begin{lemma}\label{MS}
  For each object $X\in \hcO(\chi_\mu,0)$ and each projective functor
  $\EuScript{P}\colon\hcO(\chi_{\mu'},0)\to \hcO(\chi_\mu,0)$, we have an
  inclusion $\Ann_{U(\fg)}(\Phi_{\mu}(\EuScript{P}X))\supset
  \Ann_{U(\fg)}(\Phi_{\mu'}(X))$.
\end{lemma}
\textbf{Proof:}
  The \Milicic-Soergel result depends on equivalences of categories
  \begin{equation*}
{_{\chi_\mu}\cH_{\chi_0}}\overset{\varphi_1}\longrightarrow\hcO(\chi_\mu,0)\hspace{.7in}
{_{\chi_0}\cH_{\chi_\mu}}\overset{\varphi_2}\longrightarrow \hcO(\chi_0,f)
\end{equation*}
with certain blocks of the category $\cH$ of Harish-Chandra bimodules,
defined by fixing the generalized character of the left and right
$Z$-actions. The equivalence $\Phi_\mu$ is given by composing these with
the ``flip equivalence'' $\EuScript{F}\colon
{_{\chi_\mu}\cH_{\chi_0}}\to {_{\chi_0}\cH_{\chi_\mu}}$, so we have
$\Phi_\mu=\vp_2\EuScript{F}\vp_1^{-1}$.
There are many such equivalences, one for each compatible choice of $\mu$ and $f$, but
for simplicity, omit them from the notation.

The functors $\vp_i$ are both of the form $$\vp_i(Y)=\varprojlim
Y\otimes_{U(\fg)} L_i^n$$ for an inverse system of representations
$L_i^n$ for $i=1,2$, and $n\geq 0$.  Thus, for any $U(\fg)$-bimodule
$D$, we have $\vp_i(D\otimes_{U(\fg)} Y)\cong D\otimes_{U(\fg)}
\vp_i(Y)$ and for any two-sided ideal $I$, we
have $\vp_i(IY)\cong I\vp_i(Y).$ 

For any projective functor $\EuScript{P}:\hcO(\chi_{\mu'},0)\to
\hcO(\chi_\mu,0)$ where we let $\EuScript{P}_L$ denote the functor on
$U(\fg)$-bimodules given by acting with the projective functor
$\EuScript{P}$ on the left.  By the same principle, we have
$\EuScript{P}\vp_1=\vp_1\EuScript{P}_L$.

One important consequence of this fact is that the left annihilator of
a bimodule $Y$ coincides with the annihilator of $\vp_i(Y)$. Thus,
$\vp_1^{-1}(X)$ is a Harish-Chandra bimodule whose left annihilator is
that of $X$ and whose right annihilator is that of $\Phi_\mu(X)$.
Thus, we have $$\Ann(\Phi_\mu(\EuScript{P}X))=
\operatorname{RAnn}(\EuScript{P}_L\vp_1^{-1}(X))\supset
\operatorname{RAnn}(\vp_1^{-1}(X))=\Ann(\Phi_{\mu'}(X))$$
since any element of $U(\fg)$ which right annihilates $ \vp_1^{-1}(X)$
also right annihilates $\EuScript{P}_L\vp_1^{-1}(X)$.\qed
\begin{lemma}\label{Vogan}
  If $L_1$ and $L_2$ are simple, then $L_2$ appears as a composition factor in $\EuScript{P} L_1$ for some projective functor if and only if $\Ann_{U(\fg)}(\Phi_\mu L_1)\subset \Ann_{U(\fg)}(\Phi_{\mu'} L_2)$.
\end{lemma}
\textbf{Proof:}
Consider the Harish-Chandra bimodules $\vp_1^{-1}(L_1)$ and
  $\vp_1^{-1}(L_2)$.  By reversing left and right in \cite[Theorem
  3.2]{Vogan}, we have $$\Ann(\Phi_\mu
  L_1)=\operatorname{RAnn}(\vp_1^{-1}(L_1))\subset
  \operatorname{RAnn}(\vp_1^{-1}(L_2))=\Ann(\Phi_{\mu'} L_2)$$ if and only if
  there is a projective functor $\EuScript{P}$ such that
  $\vp_1^{-1}(L_2)$ is a composition factor in
  $\EuScript{P}_L\vp_1^{-1}(L_1)$.  Applying the equivalence $\vp_1$,
  this is true if and only if $L_2$ is a composition factor of
  $\EuScript{P}L_1$.\qed

\begin{prooftr}
  If $X$ is in $ \cO^\fq(\chi_\mu,0)$, then there is a projective
  functor $\EuScript{P}\colon\cO^\fq(\chi_0,0)\to\cO^\fq(\chi_\mu,0) $
  such that there is a surjection
  $\EuScript{P}M^\fq\to X$, (this is an old result; it appears as \cite[Proposition 22]{Kho}) and thus, by Lemma \ref{MS}, we have
  inclusions $$\Ann(\Phi_0(M^\fq))\subset\Ann(\Phi_\mu(\EuScript{P}M^\fq))\subset
  \Ann(\Phi_\mu(X)).$$

  Assume $\Phi_\mu(X)$ is killed by $I_\fq$. We wish to prove that
  $X$ is $U(\fq)$-locally finite.  Since a finite-length module is $U(\fq)$-locally
  finite if and only if all its composition
  factors are, we may assume that $X$ is simple. Let $J_i$ be the
  primitive ideals killing the composition factors $L_i$ of $M^\fq$,
  in the order induced by a Jordan-H\"older series.  Then $J_1\cdots
  J_m\subset I_\fq\subset \Ann(\Phi_\mu(X))$.  Since $\Ann(\Phi_\mu(X))$ is
  prime, we have $J_i\subset \Ann(\Phi_\mu(X))$ for some $i$.  Thus, by
 Lemma \ref{Vogan}, $X$ appears as a composition
  factor of $\EuScript{P}L_i$.  Since $L_i$ is $\fq$-locally finite,
  so is $\EuScript{P}L_i$, and thus $X$.

  Finally, we establish the relationship between $\fh$-semi-simplicity
  and $Z$-semi-simplicity.  Applying \cite[Theorem 5.3]{MiS} in the case
  where $n=1$, we have that a bimodule $D\in{_{\chi_\mu}\cH_{\chi_0}}$
  is semi-simple as a right $Z$-module if and only if $\vp_1(D)$ is a
  module on which the Cartan acts semi-simply, that is, $\vp_1(D)$ is a weight
  module.  

 A bimodule $E\in{_{\chi_0}\cH_{\chi_\chi}}$ is
  semi-simple as a {\it left} $Z$-module if and only if $E$ is
  left annihilated by $\ker \chi_0$; similarly $\vp_2(E)$ is semi-simple as
  a $Z$-module if and only if $\vp_2(E)$ is annihilated by
  $\ker\chi_0$.  Since $\ker \chi_0\cdot \vp_2(E)=\vp_2(\ker
  \chi_0\cdot E),$
these conditions are equivalent.

Thus, we have a chain of equivalences
\vspace{-2.5mm}
  \begin{itemize}
\itemsep=.01\itemsep
  \item $X$ is a weight module if and only if 
    \item $\vp_1^{-1}(X)$ is
    semi-simple as a right $Z$-module if and only if
\item     $\EuScript{F}\vp_1^{-1}(X)$ is semi-simple as a left $Z$-module if
    and only if 
\item $\Phi_\mu(X)$ is semi-simple as a $Z$-module.\qed
  \end{itemize}
\end{prooftr}
\par

\subsection{Finite W-algebras}
\label{W-alg}

Now, we will connect the previous section, which only dealt with
representations of Lie algebras, to the study of finite W-algebras.

The finite W-algebra $\cW_e$ is a infinite-dimensional associative
algebra constructed from the data of $\fg$ and a nilpotent element
$e\in\fg$; geometrically, it is obtained by taking a quantization of
the Slodowy slice to the orbit through $e$.  It is most easily defined
as the endomorphism ring $\operatorname{End}_{(U(\fg)}(\sQ_{\Xi,\fm})$ of the generalized Gelfand-Graev representation
$\sQ_{\Xi,\fm}=\Xi\otimes_{U(\fm)}U(\fg)$ for a particular nilpotent
subalgebra $\fm\subset \fn$ depending on $e$, and a
one-dimensional representation  $\Xi$ of $\fm$ constructed using $e$.
We will simplify notation by omitting the subscript and simply writing
$\sQ$. The map $b:\cW_e\to
\sQ$  defined by $b(a)= a(1\otimes 1)$ is an injection,
and $\cW_e$ is often identified with this subspace in $\sQ$.
The {\bf Skryabin functor} $-\otimes_{\cW_e}\sQ:\cW_e\operatorname{-mod}\to U(\fg)\operatorname{-mod}$ is full
and faithful; its essential image is usually called the category of
{\bf Whittaker modules}.
 For
more on the general theory of W-algebras, see \cite{Sk,GG}.

The algebra $\cW_e$ has a category of representations analogous to
$\hcO$ above, originating in the work of Brundan, Goodwin and
Kleshchev \cite{BGK}.  This category is associated a {\em choice} of
parabolic $\fp$ such that
\begin{itemize}\setlength{\itemsep}{-1pt}
\item 
  $e$ is distinguished in the Levi $\fl$ of $\fp$ and
\item  $\fp$ contains a fixed
  maximal torus $\ft$ of the centralizer $C_{\fg}(e)=\{x\in
  \fg|[x,e]=0\}$.
\end{itemize}
This choice of parabolic allows us
to put an preorder on the weights of $\ft$ by declaring that $\la\leq
\mu$ if and only if $\mu-\la$ is a linear combination of weights of
$\ft$ acting on $\fp$.

We have a natural inclusion $U(\ft)\hookrightarrow \cW_e$, (described
in \cite[Theorem 3.3]{BGK}) and thus can decompose any $\cW_e$
representation by its generalized weight spaces with respect to $\ft$.

\begin{definition}Let $\hcO(\cW_e,\fp)$ be the category of modules $X$ such that
\begin{itemize}\setlength{\itemsep}{-1pt}
  \item $\ft$ acts on $X$ with finite dimensional generalized weight spaces.
  \item the weights which appear in $X$ are contained in a finite
    union of sets of the form $\{\mu\vert \mu\leq \la\}$.
\end{itemize}
As before, $\cO(\cW_e,\fp)$ denotes the subcategory on which $\ft$
acts semi-simply, and $\pcO(\cW_e,\fp)$ denotes the subcategory on
which $Z\cong Z(\cW_e)$ acts semi-simply.
\end{definition}

Alternatively, following \cite{LosO}, we could define the category
$\hcO(\cW_e,\fp)$ as the category of $\cW_e$-modules which are locally finite for
the action of the subalgebra $\cW_e^+\subset \cW_e$ of non-negative weight spaces of an element $\xi\in
\ft$ acting by the adjoint action.  The parabolic $\fp$ derived from
$\xi$ is precisely the non-negative
weight spaces for the adjoint action of $\xi$ on $\fg$.  In Losev's
notation, $\hcO(\cW_e,\fp)$ is denoted $\cO(\xi)$, and the
category we and \cite{BGK} denote $\cO(\cW_e,\fp)$, Losev would denote
$\cO^{\ft}(\xi)$ (leaving the nilpotent $e$ implicit).

There is an isomorphism $Z\cong Z(\cW_e)$, typically referenced in the
literature to a footnote to Question 5.1 in the paper of Premet
\cite{Pjos}, where Premet gives an argument he ascribes to
Ginzburg. This isomorphism allows us to identify central characters of
$U(\fg)$ and $\cW_e$.
\begin{definition}
  We denote the subcategory of $\hcO(\cW_e,\fp)$ with generalized
  central character $\chi$ by $\hcO(\chi,\cW_e,\fp)$, and the
  subcategory where the center $Z$ acts semi-simply (and thus
  according to the character
  $\chi$) by $\pcO(\chi,\cW_e,\fp)$. The category $\mathcal{C}$
  defined in the introduction in this notation would be written
  $\pcO(\chi_0,\cW_e,\fp)$.
\end{definition}

We will primarily interested in the case where after fixing a
parabolic $\fp\supset \fb$ with Levi $\fl$, we take $e$ is a regular
nilpotent in $\fl$ (for the Borel given by intersection with $\fb$),
and we will assume we are in this situation.  In type A, all
nilpotents will appear this way, since we can take
$\fp$ to be block upper triangular matrices attached to the Jordan
blocks of $e$; in other types, there are nilpotents that are not of
this form.

Ideals of $U(\fg)$ and of $\cW_e$ are related by a pair of maps: $J\mapsto
J^\dagger$ sends two-sided ideals of $\cW_e$ to two-sided
ideals of $U(\fg)$ and $I\mapsto I_\dagger$ goes the
opposite direction; these both are defined by Losev
in \cite[\S 3.4]{LosW}.  
  As described in \cite[\S 3.5]{Losfin}, the ideal $I_\dagger$ is the preimage $b^{-1}(\sQ J)$.
  Alternatively, it is given by $\mathrm{Wh}^\fm_\fm(I)=\Hom(\sQ,\sQ\otimes_{U(\fq)}I)$, the biWhittaker functor of Ginzburg
  (\cite[(3.3.2)]{Gin08}) applied to $J$; similarly, $\mathrm{Wh}^\fm_\fm(U(\fg)/I)=\cW_e/I_\dagger$.  In particular, $\cW_e/I_\dagger$ is
  the non-commutative Hamiltonian reduction of $U(\fg)/I$ by the
  action of $\fm$ with respect to the 1-dimensional representation $\Xi$.

As before, let $f\colon\fn\to\C$ be a fixed character of this
nilpotent Lie algebra, and we define a nilpotent element
$e=e_f:=\sum_{i\in \Delta_f}E_i.$ This a regular nilpotent in the Levi
$\fl_f$ of $\fp_f$.
\begin{proposition}[\cite{LosO}]\label{Losev}
 There is an equivalence
  $$\cL\colon \hcO(\chi_0,\cW_{e},\fp_f)\to \hcO(\chi_0,f)$$ inducing equivalences
  \begin{equation*}
    \pcO(\chi_0,\cW_e,\fp_f)\cong \pcO(\chi_0,f) \hspace{.6in}  \cO(\chi_0,\cW_e,\fp_f)\cong \cO(\chi_0,f)
  \end{equation*}
  Furthermore, $\cL(X)$ is killed by $I_\fq$ if and only if $X$ is 
  killed by $J_\fq:= (I_\fq)_{\dagger}\subset \cW_e$.
\end{proposition}
We denote the subcategory of ideals killed by $J_\fq$ by $\hcO_\fq(\chi_0,\cW_e,\fp_f)$.

\bigskip

\begin{proof}
  The only part of the equivalences not stated directly in \cite[Theorem 4.1]{LosO} is the
  equivalence of most interest for us, that of
  $\pcO(\chi_0,\cW_e,\fp_f)\cong \pcO(\chi_0,f)$.  

By \cite[Theorem 4.1(1)]{LosW}, we know that
  $\Ann(\cL X)=\Ann(X) ^\dagger$; thus under the isomorphism
  $Z(\cW_e)\cong Z$, the intersections $\Ann(\cL X) \cap
  Z$ and $\Ann(X)\cap Z(\cW_e)$ are identified by \cite[Theorem
  1.2.2(iii)]{LosW}.  This shows that
\begin{center}
$(*)$: the center of $\cW_e$ acts semi-simply on $X\in
  \hcO(\chi_0,\cW_e,\fp_f)$ if and only if the center of $U(\fg)$ acts semi-simply on
  $\cL(X)$.
\end{center} 
 The ``only if'' portion of $(*)$ shows that $\cL$ induces a
  fully faithful functor $\cL:\pcO(\chi_0,\cW_e,\fp_f) \to
  \pcO(\chi_0,f)$. The ``if'' portion of $(*)$ shows that the
  essential surjectivity of $\cL$ implies the essential surjectivity
  of this restriction.  Thus, we have the desired equivalence.

\medskip

  Now, fix a $U(\fg)$-module $Y$ in $\hcO(\chi_0,f)$ such that
  $\Ann(Y)\supset I_\fq$; we wish to establish that $\cL^{-1}Y$ is
  killed by $J_\fq$.  We have already noted that
  $\Ann(Y)=\Ann(\cL^{-1}Y)^\dagger$.  It follows that
  \[\Ann(\cL^{-1}Y)\supset( (\Ann(\cL^{-1}Y))^\dagger)_\dagger=\Ann(Y)_\dagger\] by \cite[Theorem
  3.4.4]{LosW}. Furthermore, $\Ann(Y)_\dagger \supset
  J_\fq$ since $(\cdot)_\dagger$ preserves inclusion, as is noted by Losev in \cite[\S 3.4]{LosW}. 

  On the hand, if $X\in\pcO(\chi_0,\cW_e,\fp_f)$ is killed by $J_\fq$
  then $\cL X$ is a $U(\fg)$-module killed by $(J_\fq)^\dagger$, since
  $(\cdot)^\dagger$ preserves inclusion by \cite[1.2.2(i)]{LosW}
  applied to the two ideals being compared.  	Using
  \cite[Proposition 3.4.4]{LosW} again,
  $(J_\fq)^\dagger=((I_\fq)_\dagger)^\dagger\supset I_\fq$ and $\cL X$ is killed by $I_\fq$.
\end{proof}

To recap, if $\mu$ is an integral weight of $\fg$ such that $\mu+\rho$ is
dominant integral, and $e$ is the principal nilpotent of the Levi in
the parabolic $\fp_f$ corresponding to the stabilizer of $\mu+\rho$
under the usual action of $W$ on weights, then combining Propositions
\ref{trans} and \ref{Losev}, we obtain our main theorem.

\begin{theorem}\label{main}
  There is an equivalence $\Psi=\cL^{-1}\circ \Phi_\mu\colon
  \cO^\fq(\chi_\mu,0)\cong \pcO_\fq(\chi_0,\cW_e,\fp_f)$.
\end{theorem}

\subsection{Relationship to geometry}
\label{geometry}
In order to give the reader some feeling for the ``meaning'' of the
ideal $I_\fq$, we will briefly indicate its relationship to geometry.

We have the $G$-space $G/Q$, and as with any $G$-space, a map
$\a\colon U(\fg)\to D(G/Q)$ to the ring of global differential
operators on $G/Q$ sending a Lie algebra element to the corresponding
vector field. For $i$ in the Dynkin diagram of $\fg$, we let $\bar i$
be the unique node such that $\al_i=-w_0(\al_{\bar i})$.  This Dynkin
diagram automorphism induces an automorphism $\tau\colon \fg\to \fg$
sending $F_i\mapsto F_{\bar i}, E_i\mapsto E_{\bar i}$.  We let
$\bar{\fq},\bar{Q}$ be the parabolic subalgebra and subgroup which are
the image of $\fq$ and $Q$ under this automorphism.

\begin{proposition}\label{parab-iso}
  The map $\a\circ \tau$ induces an isomorphism $D(G/Q)\cong
  U(\fg)/I_\fq$.  In
  particular, in the induced order filtration on $U(\fg)/I_\fq$, we
  have an isomorphism $\gr (U(\fg)/I_\fq)=\C[T^*G/P]$, and $I_\fq$ is prime.
\end{proposition}
\begin{proof}
 Borho and Brylinski \cite[3.8]{BoBrI} show\footnote{We thank the referee
    for pointing out this reference.} that $\ker \a$
  is the annihilator of a simple parabolic Verma module
  $L^\fq$ with highest weight $2\rho_Q-2\rho=w_0^Qw_0(\rho)-\rho$,
  where $w_0^Q$ is the longest element in $W_Q\subset W$; this is the
  unique simple parabolic Verma module 
  with central character $\chi_0$.  If we twist the action of $\fg$ on
  $L^\fq$ by $\tau$, we obtain $L^{\bar \fq}$, the simple with highest
  weight $2\rho_{\bar Q}-2\rho=w_0w_0^Q(\rho)-\rho$.  We let $I:=\Ann(L^{\bar \fq})=\ker\a\circ \tau$. Thus, we need only show that
  $I=I_\fq$.

It is a theorem of Irving \cite[\S 4.3]{Irv85}
  that any two simples which both appear in the socle of a parabolic
  Verma module for $\fq$ are in the same right cell.  In particular, the
  socle of $M^\fq$ is in the same right cell as $L^\fq$.  Thus, the
  socle of $\Phi_0(M^\fq)$ is in the same left cell as $L^{\bar \fq}$, and so
  these simple modules have the same annihilator.  Since any ring
  element annihilating $\Phi_0(M^\fq)$ annihilates its socle, we have
  that $I_\fq\subset I$. 

Let $d(M)$ denote the {\bf Gelfand-Kirillov dimension} of a
$\fg$-module $M$; this is the degree of the Hilbert polynomial of the
associated graded of $M$ in any good filtration.  We now recall
certain facts from \cite{Irv85}:
\begin{itemize}
\item The socle of $M^\fq$ is irreducible \cite[\S 4.1]{Irv85}; thus,
  the socle of $\Phi_0(M^\fq)$ is as well.
\item A simple in $\cO^\fq(\chi_0,0)$ has Gelfand-Kirillov dimension equal or greater to
  $d(\soc M^\fq)$ if and only if it lies in the same right cell as  $\soc M^\fq$.
  \cite[\S 4.3, (iv) \& (v)]{Irv85}; again, the same
  holds for $\Phi_0(M^\fq)$ since $\Phi_0$ preserves G-K dimension.
\item If a simple in the  right cell as  $\soc M^\fq$ appears as a
  composition factor the parabolic Verma module of highest weight
  $w\rho-\rho$, then it is the socle and appears nowhere
  else in the composition series of a parabolic Verma
  module of highest weight $w'\!\rho-\rho$ with $w'\leq w$ in
  Bruhat order \cite[\S 4.6, Proposition 1]{Irv85}. In particular, so
  such composition factor can occur in $M^\fq/\soc M^\fq$.
\end{itemize}
Since $\Phi_0$ preserves Gelfand-Kirillov dimension, we have that 
\begin{itemize}
\item[$(*)$] No composition factors of $\Phi_0(M^\fq)/\soc \Phi_0(M^\fq)$ have Gelfand-Kirillov
  dimension greater than or equal to  $d(\soc \Phi_0(M^\fq))$.  
\end{itemize}
Now, consider the multiplication map \[m\colon I \otimes_{U(\fg)}
\Phi_0(M^\fq)\to \Phi_0(M^\fq).\]  Since $I$ kills $\soc
\Phi_0(M^\fq)$, this factors through a map \[\tilde m\colon I\otimes_{U(\fg)} (
\Phi_0(M^\fq)/\soc  \Phi_0(M^\fq))\to  \Phi_0(M^\fq).\]  By $(*)$, we
know that
$d(\Phi_0(M^\fq)/\soc \Phi_0(M^\fq)) < d(\soc \Phi_0(M^\fq)) $. Tensor
product with a Harish-Chandra bimodule can only decrease G-K
dimension, so we have that \[d(\im\tilde m)\leq d(\Phi_0(M^\fq)/\soc  \Phi_0(M^\fq)) < d(\soc \Phi_0(M^\fq)) .\]  If $\im \tilde m\neq 0$,
then it must contain $\soc \Phi_0(M^\fq)$, which is impossible by the
dimension equality.  Thus $\tilde m=0$ and consequently, $m=0$.  That
is to say, $I$ annihilates $\Phi_0(M^\fq)$, and so $I_\fq=I$.

  This isomorphism, followed by the symbol map on differential
  operators induces an isomorphism $\gr (U(\fg)/I_\fq)\cong \C[T^*G/Q]$.
  Since the associated graded of $U(\fg)/I_\fq$ is the global
  function ring of an irreducible quasi-projective variety, the
  associated graded is an integral domain, and thus so is
  $U(\fg)/I_\fq$.
\end{proof}

As discussed by Borho and
Brylinski, the filtration induced on $U(\fg)/I_\fq$ by this
isomorphism may or may not coincide with that induced
by the usual PBW filtration on $U(\fg)$; see \cite[Theorem
5.6]{BoBrI}.  This the order filtration, in turn, induces a good filtration on
$\cW_e/J_\fq$.

The associated graded of the algebra $\cW_e/J_\fq$ this respect to
this filtration also has a geometric description,
though it is a bit less familiar.  Let $\cS\subset \fg\cong \fg^*$ denote the
{\bf Slodowy slice} $e+\ker \mathrm{ad}_f$ where $e$ and $f$ generate an
$\mathfrak{sl}_2$ as the Chevalley generators.  We let
$\pi_Q:T^*G/Q\to \fg^*$ be the canonical moment map for the
$G$-action; this is an analogue of the Springer map defined by
identifying the cotangent fiber over the coset $gQ$ with $(\mathrm{Ad}_g\fq
)^\perp$.   As in the introduction, we let $\fX^Q_e$ denote the fiber product
$T^*G/Q\times_{\fg^*}\cS$. By \cite[Corollary 1.3.8]{Gin08} applied
to the map $\pi_Q$, this variety is smooth of dimension $\dim
T^*G/Q-\dim G\cdot e$, and it is symplectic, since it can also be
described as the symplectic reduction of $T^*G/Q$ by the action of
connected subgroup $M\subset G$ with Lie algebra $\fm$.  

\begin{proposition}\label{W-ass}
  The associated graded of $\cW_e/J_\fq$ is isomorphic to
  $\C[\fX^Q_e]$, the ring of global functions on $\fX^Q_e$.  In
  particular, $J_\fq$ is prime and $J_\fq^\dagger=I_\fq$.  The
  Skryabin functor defines an equivalence between $\cW_e/J_\fq$-modules and
  Whittaker modules over $U(\fg)/I_\fq$.
\end{proposition}
\begin{proof}
  We apply \cite[Theorem 4.1.4(i)]{Gin08} to the
  bimodule $U(\fg)/I_\fq$; the associated graded of $U(\fg)/I_\fq$ is
  the ring of global functions $\C[T^*G/Q]$, so the associated graded
  of $\cW_e/J_\fq=\mathrm{Wh}^\fm_\fm(U(\fg)/I_\fq)$ is
  $\C[T^*G/Q]\otimes_{\C[\fg^*]}\C[\cS]\cong \C[\fX^Q_e]$.  Since this
  is the ring of functions on a smooth quasi-projective variety, it is
  a domain, and so is $\cW_e/J_\fq$.  

Finally, note that $V(\gr J_\fq)=
  (G\cdot \fq^\perp\cap \cS)$ and so $J_\fq$ is admissible.   Since $J_\fq$ is minimal over itself, 
  \cite[1.2.2(vii)]{LosW} implies that $J_\fq^\dagger=I_\fq$.  Given
  any module $M$ over $\cW_e$ such that $\Ann(M)\supset J_\fq$, we
  have that $\Ann(M\otimes_{\cW_e}\sQ)=\Ann(M)^\dagger \supset I_\fq$,
  and {\it vice versa} for any Whittaker module killed by $I_\fq$.
\end{proof}
\begin{remark}
  In fact, the algebras $U(\fq)/I_\fq$ and $\cW_e/J_\fq$
  have deeper ties to geometry than the theorems above make clear.
  There is a ``geometric category $\cO$''
  attached to a symplectic variety $X$ and a Hamiltonian $\C^*$-action on $X$
  (which plays the role of choosing a Borel for usual category
  $\cO$).  This is a category of sheaves of modules over a
  quantization of the structure sheaf on $X$ (in the sense of
  \cite[Definition 1.3]{BK04a}) satisfying certain geometric
  properties.

  When applied to the varieties $\fX^Q_e$, the categories we arrive at
  exactly $\cO^\fq(\chi_\mu,0)$, and the equivalence is given by
  applying a (modified) sections functor to arrive in the category of
  modules over $\cW_e/J_\fq$, and then applying the equivalence of
  Theorem \ref{main}. A precise description of this construction
  requires an investment in technology which there is no space for in
  this paper, and will be given in a forthcoming paper of the author,
  Braden, Licata and Proudfoot \cite{BLPWgco}; much of the set-up is
  done in the special case of hypertoric varieties (another class of
  symplectic varieties) in \cite[\S 5]{BLPWtorico}.
\end{remark}

Interestingly, this establishes a conjecture which had appeared in an
unpublished manuscript of Bezrukavnikov and \Mirkovic, and which had
been previously studied by the author and Stroppel.
\begin{definition}
  We say that a $D$-module $\cM$ on $G/Q$ is {\bf $\mathbf{(N,f)}$-equivariant} if
  the action of $\fn$ on $\cM$ by $n\cdot m=\a_nm-f(n)m$ (where $\a_n$
  is the vector field given by the infinitesimal action of $n$ on
  $G/Q$) integrates to an $N$-equivariant structure on $\cM$.  Some
  authors use ``strongly equivariant'' for this form of equivariance.
\end{definition}
In the case where $Q=B$, these D-modules are studied by Ginzburg in
\cite[\S 5]{Gin08}, where he calls them ``Whittaker D-modules.''

There is also a version of $(N,f)$-equivariance for perverse sheaves
on $G/Q$, though in the algebraic category this can only be made sense
of using perverse sheaves on the characteristic $p$ analogue of $G/Q$ (for more
on this approach, see \cite{BBM}). In this case, we let $f'$ be a character $f'\colon N/\mathbb{F}_p\to
\mathbb{G}_a/\mathbb{F}_p$ such that $\Delta_{f'}=\Delta_{f}$ and
consider an Artin-Schreier sheaf pulled back from
$\mathbb{G}_a/\mathbb{F}_p$ by $f'$, which we denote by $\mathcal A$. 
A $(N,\mathcal A)$-equivariant sheaf is a pair of a sheaf $\cG$
together with an isomorphism $a^*\cG\cong \mathcal{A}\boxtimes \cG$ of
sheaves on $N\times G/Q$, where $a\colon N\times G/Q\to G/Q$ is the
action map.  Of course, this map must satisfy an associativity
property as with usual equivariance.  
 Such sheaves are sometimes called
``partial Whittaker'' or ``Whittavariant'' sheaves and
are discussed in detail in \cite[\S 3.1]{BY}.

\begin{corollary}
  The category $\cO^\fq(\chi_\mu,0)$ is equivalent to the category of
  $(N,f)$-equivariant $D$-modules on $G/Q$ (and thus also to the
  category of $(N,f')$-equivariant perverse sheaves on $G/Q$, with all
  groups considered as algebraic groups over $\mathbb{F}_p$).
\end{corollary}
\begin{proof}
  By Beilinson-Bernstein \cite{BB}, the sections functor
  $\cD_{G/Q}-\mathsf{mod}\to U(\fg)/I_\fq-\mathsf{mod}$ is an
  equivalence.  If $\cM$ is a $(N,f)$-equivariant $D$-module, then the
  action of $U(\fn)$ on $\Gamma(\cM)$ is locally finite, with $\C_f$
  being the only representation which appears in the composition
  series, so $M\in\pcO_{\fq}(\chi_0,f)$.  On the other hand, if
  $M\in\pcO_{\fq}(\chi_0,f)$, then the $f$-twisted $\fn$-action via
  differential operators on the localization will integrate to an
  $(N,f)$-action.
\end{proof}

\begin{remark}
  As we noted earlier, the $B=Q$ case of this can be derived from the
  ``paradromic/Whitta\-var\-iant'' duality of \cite[Theorem 4.4.1]{BY}.
  Bezrukavnikov and Yun establish a Koszul duality between the
  category of $(N,f')$-equivariant perverse sheaves on $G/B$ (in their
  notation, $B$ is denoted $I$) and the category of $N$-equivariant
  perverse sheaves on $G/P_{f'}$, which is equivalent to
  $\cO^{\fp_{f'}}_0$ by \cite[3.5.3]{BGS96}.  That the Koszul dual of
  the latter category is a singular block of category $\cO$ is the
  main theorem of \cite{BGS96} (Theorem 1.1.3).  Stringing together
  these equivalences gives that between a singular block of category
  $\cO$ and Whittaker sheaves on $G/B$ for $f'$.
\end{remark}

 \bibliography{./symplectic}

\def\cprime{$'$}
\providecommand{\bysame}{\leavevmode\hbox to3em{\hrulefill}\thinspace}
\providecommand{\MR}{\relax\ifhmode\unskip\space\fi MR }
\providecommand{\MRhref}[2]{%
  \href{http://www.ams.org/mathscinet-getitem?mr=#1}{#2}
}
\providecommand{\href}[2]{#2}
\begin{thebibliography}{BLPWb}

\bibitem[BB81]{BB}
Alexandre Be{\u\i}linson and Joseph Bernstein, \emph{Localisation de
  {$\mathfrak g$}-modules}, C. R. Acad. Sci. Paris S\'er. I Math. \textbf{292}
  (1981), no.~1, 15--18.

\bibitem[BB82]{BoBrI}
Walter Borho and Jean-Luc Brylinski, \emph{Differential operators on
  homogeneous spaces. {I}. {I}rreducibility of the associated variety for
  annihilators of induced modules}, Invent. Math. \textbf{69} (1982), no.~3,
  437--476.

\bibitem[BBM04]{BBM}
Roman Bezrukavnikov, Alexander Braverman, and Ivan Mirkovic, \emph{Some results
  about geometric {W}hittaker model}, Adv. Math. \textbf{186} (2004), no.~1,
  143--152.

\bibitem[BGG76]{BGG}
I.~N. Bern{\v{s}}te{\u\i}n, I.~M. Gel{\cprime}fand, and S.~I. Gel{\cprime}fand,
  \emph{A certain category of {${\mathfrak g}$}-modules}, Funkcional. Anal. i
  Prilo\v zen. \textbf{10} (1976), no.~2, 1--8.

\bibitem[BGK08]{BGK}
Jonathan Brundan, Simon~M. Goodwin, and Alexander Kleshchev, \emph{Highest
  weight theory for finite {$W$}-algebras}, Int. Math. Res. Not. IMRN (2008),
  no.~15, Art. ID rnn051, 53.

\bibitem[BGS96]{BGS96}
Alexander Beilinson, Victor Ginzburg, and Wolfgang Soergel, \emph{Koszul
  duality patterns in representation theory}, J. Amer. Math. Soc. \textbf{9}
  (1996), no.~2, 473--527.

\bibitem[BK04]{BK04a}
Roman Bezrukavnikov and Dmitry Kaledin, \emph{Fedosov quantization in algebraic
  context}, Mosc. Math. J. \textbf{4} (2004), no.~3, 559--592, 782.
  \MR{MR2119140 (2006j:53130)}

\bibitem[BLPWa]{BLPWgco}
Tom Braden, Anthony Licata, Nicholas Proudfoot, and Ben Webster,
  \emph{Algebraic and geometric category $\mathcal{O}$}, in preparation.

\bibitem[BLPWb]{BLPWtorico}
\bysame, \emph{Hypertoric category $\mathcal{O}$}, \textsf{arXiv:1010.2001}.

\bibitem[Bru08a]{Bru06}
Jon Brundan, \emph{Symmetric functions, parabolic category $\mathcal{O}$ and
  the {S}pringer fiber}, Duke Math. J. \textbf{143} (2008), 41--79.

\bibitem[Bru08b]{Bru08}
Jonathan Brundan, \emph{Centers of degenerate cyclotomic {H}ecke algebras and
  parabolic category {$\mathcal O$}}, Represent. Theory \textbf{12} (2008),
  236--259.

\bibitem[BY]{BY}
Roman Bezrukavnikov and Zhiwei Yun, \emph{On the {K}oszul duality for affine
  {K}ac-{M}oody groups}, \textsf{arXiv:1101.1253}.

\bibitem[GG02]{GG}
Wee~Liang Gan and Victor Ginzburg, \emph{Quantization of {S}lodowy slices},
  Int. Math. Res. Not. (2002), no.~5, 243--255.

\bibitem[Gin]{Gin08}
Victor Ginzburg, \emph{Harish-{C}handra bimodules for quantized {S}lodowy
  slices}, \textsf{arXiv:0807.0339}.

\bibitem[HTT08]{HTT}
Ryoshi Hotta, Kiyoshi Takeuchi, and Toshiyuki Tanisaki, \emph{{$D$}-modules,
  perverse sheaves, and representation theory}, Progress in Mathematics, vol.
  236, Birkh\"auser Boston Inc., Boston, MA, 2008, Translated from the 1995
  Japanese edition by Takeuchi.

\bibitem[Irv85]{Irv85}
Ronald~S. Irving, \emph{Projective modules in the category {${\mathcal O}\sb
  S$}: self-duality}, Trans. Amer. Math. Soc. \textbf{291} (1985), no.~2,
  701--732.

\bibitem[Kho05]{Kho}
Oleksandr Khomenko, \emph{Categories with projective functors}, Proc. London
  Math. Soc. (3) \textbf{90} (2005), no.~3, 711--737.

\bibitem[Kos78]{Kos}
Bertram Kostant, \emph{On {W}hittaker vectors and representation theory},
  Invent. Math. \textbf{48} (1978), no.~2, 101--184.

\bibitem[Losa]{Losfin}
Ivan Losev, \emph{Finite dimensional representations of {W}-algebras},
  \textsf{arXiv:0807.1023}.

\bibitem[Losb]{LosO}
\bysame, \emph{On the structure of the category ${\cO}$ for {W}-algebras},
  \textsf{arXiv:0812.1584}.

\bibitem[Losc]{LosW}
\bysame, \emph{Quantized symplectic actions and {W}-algebras},
  \textsf{arXiv:0707.3108}.

\bibitem[MS97]{MiS}
Dragan Mili{\v{c}}i{\'c} and Wolfgang Soergel, \emph{The composition series of
  modules induced from {W}hittaker modules}, Comment. Math. Helv. \textbf{72}
  (1997), no.~4, 503--520.

\bibitem[Pre02]{Sk}
Alexander Premet, \emph{Special transverse slices and their enveloping
  algebras}, Adv. Math. \textbf{170} (2002), no.~1, 1--55, With an appendix by
  Serge Skryabin.

\bibitem[Pre07]{Pjos}
\bysame, \emph{Enveloping algebras of {S}lodowy slices and the {J}oseph ideal},
  J. Eur. Math. Soc. (JEMS) \textbf{9} (2007), no.~3, 487--543.

\bibitem[Soe86]{Soe86}
Wolfgang Soergel, \emph{\'{E}quivalences de certaines cat\'egories de
  {${\mathfrak g}$}-modules}, C. R. Acad. Sci. Paris S\'er. I Math.
  \textbf{303} (1986), no.~15, 725--728.

\bibitem[Str]{Str06b}
Catharina Stroppel, \emph{{Parabolic category $\mathcal{O}$, Perverse sheaves
  on Grassmannians, Springer fibres and Khovanov homology}},
  \textsf{arXiv:math.RT/0608234}.

\bibitem[Vog80]{Vogan}
David~A. Vogan, Jr., \emph{Ordering of the primitive spectrum of a semisimple
  {L}ie algebra}, Math. Ann. \textbf{248} (1980), no.~3, 195--203.

\end{thebibliography}
\bibliographystyle{amsalpha}
\end{document}